\documentclass{rmmcart}

\usepackage{amsmath,amsthm,amssymb,latexsym,tikz}
\usetikzlibrary{decorations.markings}

\newtheorem{thm}{Theorem}
\newtheorem{lem}{Lemma}
\newtheorem{cor}{Corollary}
\newtheorem{prop}{Proposition}
\newtheorem{rem}{Remark}
\newtheorem{defin}{Definition}
\newtheorem{ex}{Example}

\newcommand{\Z}{\mathbb Z}
\newcommand{\Q}{\mathbb Q}
\newcommand{\R}{\mathbb R}
\newcommand{\C}{\mathbb C}

\DeclareMathOperator{\abs}{abs}

\newcommand{\dynkinradius}{.08cm}
\newcommand{\dynkinstep}{.75cm}
\newcommand{\dynkindot}[2]{\fill (\dynkinstep*#1,\dynkinstep*#2) circle (\dynkinradius);}
\newcommand{\dynkinXsize}{1.5}
\newcommand{\dynkincross}[2]{
\draw[thick] (#1*\dynkinstep-\dynkinXsize,#2*\dynkinstep-\dynkinXsize) -- (#1*\dynkinstep+\dynkinXsize,#2*\dynkinstep+\dynkinXsize);
\draw[thick] (#1*\dynkinstep-\dynkinXsize,#2*\dynkinstep+\dynkinXsize) -- (#1*\dynkinstep+\dynkinXsize,#2*\dynkinstep-\dynkinXsize);
}
\newcommand{\dynkinline}[4]{\draw[thin] (\dynkinstep*#1,\dynkinstep*#2) -- (\dynkinstep*#3,\dynkinstep*#4);}
\newcommand{\dynkindots}[4]{\draw[dotted] (\dynkinstep*#1,\dynkinstep*#2) -- (\dynkinstep*#3,\dynkinstep*#4);}

\newenvironment{dynkin}{\begin{tikzpicture}[decoration={markings,mark=at position 0.7 with {\arrow{>}}}]}
{\end{tikzpicture}}

\title[A classification of small operators using graph theory]{A classification of small operators using graph theory}

\author{Terrence Bisson}
\address{Department of Mathematics and Statistics,
Canisius College, 2001 Main Street, Buffalo, NY 14208}
\email{bisson@canisius.edu}

\author{Jonathan Lopez}
\address{Department of Mathematics and Statistics,
Canisius College, 2001 Main Street, Buffalo, NY 14208}
\email{lopez11@canisius.edu}

\date{August, 21, 2017}


\begin{document}

\begin{abstract}
Given a real $n \times m$ matrix $B$, its operator norm can be defined as 
$$|B|=\max_{|v|=1}|Bv|.$$
We consider a matrix ``small'' if it has non-negative integer entries 
and its operator norm is less than $2$.
These matrices correspond to bipartite graphs with spectral radius less than $2$, 
which can be classified as disjoint unions of Coxeter graphs.  This gives a direct route to an $ADE$-classification result in terms of very basic mathematical objects.
Our goal here is to see these 
results as part of 
a general program of classification of small objects, 
relating quadratic forms, reflection groups, root systems, and Lie algebras.
\end{abstract}

\maketitle

\section{Classifications}

\medskip
In mathematics, a classification result describes all possible structures of a given type, usually by 
showing that every structure 
is equivalent to one which decomposes 
into ``components'', each equivalent to one from a set of basic types.


Good examples of classification in mathematics are rare and interesting.
The  description of the possible structures of semi-simple Lie algebras over the complex numbers
is one of the most important examples.  
Every such Lie algebra is isomorphic to a direct sum of simple types, known by the alphabet
$A$, $B$, $C$, $D$, $E$, $F$, $G$ (see Fulton and Harris \cite{FH}, for instance).  
The $ADE$ part of this classification has surprising 
similarities with classification results
in many other areas of mathematics (see, for example, Hazewinkel, et al. \cite{HHSV}).
For instance, the $ADE$-series appears in Coxeter's classification of the simply-laced crystallographic finite reflection groups (see \cite{Cox}); and Cameron, et al. \cite{cameron} showed that the classification of certain sets of lines in $\R^n$ at mutual angles of $60^{\circ}$ or $90^{\circ}$ also involves the $ADE$-series (see Theorem 3.5 in \cite{cameron}).
These results have led to the development of large areas of ongoing research.


In \cite{symposium}, Arnol'd asked if the appearance of the $ADE$-series in these classifications was merely coincidence, or if there was some profound underlying cause.  Though we do not claim to provide an answer to Arnol'd's question, we do exposit a direct route to an $ADE$-classification result in terms of very basic mathematical objects.  Namely, the non-negative integer matrices with operator norm less than $2$ can be classified by the $ADE$-series of graphs.  We hope the simplicity of development will appeal to a wide audience, since the usual paths one takes to arrive at an $ADE$-classification can be long and difficult.

More generally, we want to recommend some notions of smallness in various parts of mathematics,
and show how such ``small'' objects can be classified.  In 
our examples, the small objects satisfy a quantitative bound and are
defined over the non-negative integers.

In Section 2, we define operator norm and record some results concerning the operator norm needed later in the paper.  In particular, we explain how the operator norm of a rectangular matrix can be determined from an associated square symmetric matrix and we note that the norm of a square symmetric matrix is equal to its spectral radius.  In Section 3, we prove some results concerning the Perron-Frobenius theory of non-negative square matrices.  Since we are interested only in symmetric matrices, some of the proofs differ from the standard proofs, and are simpler (and we hope more intuitive).  In Section 4, we explain the connection between small operators and small graphs, and show how small operators can be classified by the Coxeter graphs associated to the $ADE$ series.  In particular, any small operator can be represented by a disjoint union of the Coxeter graphs for the $ADE$ series.  In the final section, we sketch related ideas of smallness for
quadratic forms, reflection groups, and root systems, 
using graphs as an organizing principle.

Our ideas here are inspired by a very interesting chapter in the monograph 
by Goodman, de la Harpe, and Jones \cite{GdlaHJ}.

\section{Operators}

By an operator we mean a linear transformation $B:\R^n\to \R^m$.  So an operator $B$ can be represented by a rectangular matrix.

\begin{defin}
Let $B$ be a rectangular real matrix, i.e., an operator $\R^n\to\R^m$.
The operator norm $|B|$ can be defined
as $$|B|={\rm max}_{v\neq 0}{\frac{|Bv|}{|v|}}={\rm max}_{|v|=1}|Bv|.$$
\end{defin}

Let's consider matrices with non-negative integer entries; they are just
the finite sums of the basic operators $E_{i,j}$ 
(the rank 1 operators, the matrices with all entries $0$ except for 
just one $1$ in the $i,j$ position).
 
\begin{defin}
An operator $B$ is considered small if it has non-negative integer entries and $|B|<2$.
\end{defin}

\begin{rem}
In the final section we will 
give some indications of why $2$ is a natural and useful bound here.
\end{rem}

Our goal in this paper is to describe the classification of all small operators.
The appropriate notion of equivalence is given by the action of 
symmetric groups on the entries of vectors in $\R^m$ and $\R^n$.
The appropriate notion of decomposition corresponds to direct sum of matrices 
(juxtaposition of blocks).

\begin{ex}
Each $|E_{i,j}|=1$, so these are small operators.
\end{ex}

Note that a small operator can only have $0, 1$ entries.

\begin{ex}
Which $3\times 2$ matrices are small?    Consider all the $3\times 2$  matrices with $0,1$ entries; there are $64$ of these and $54$ of them are small. Each non-small example is 
equivalent to one of the following: 
$$\begin{pmatrix}1&1\\1& 1\\ 0&0 \end{pmatrix}\quad\begin{pmatrix}1&1\\1& 1\\ 1&0 \end{pmatrix}
\quad \begin{pmatrix}1&1\\1& 1\\ 1&1 \end{pmatrix}\ $$
\end{ex}

The definition of operator norm, 
maximizing a continuous function over a compact set,
seems based on real analysis.
For instance, we could use Lagrange multipliers to calculate the operator norm.
But for our purposes it is convenient to work with an associated symmetric matrix.

\bigskip
For any operator with matrix $B$ we have symmetric matrices $BB^\top$ and $B^\top B$;
they have  real eigenvalues (since they are symmetric)
and their eigenvalues are non-negative, since $(B^\top B)v=av$ with $v\neq0$ implies
$$av^\top v=v^\top(B^\top B)v=(Bv)^\top(Bv),$$
so that $a$ times a positive number is non-negative; and similarly for $BB^\top$.

This is related to the following construction.
Given an operator with matrix $B\in{\rm Mat}_{m,n}(\R)$, we may form a symmetric matrix $A\in{\rm Mat}_{m+n}(\R)$ by
$$A=\begin{pmatrix}0&B\\B^\top& 0\\ \end{pmatrix}.$$
We will refer to $A$ as the symmetric matrix associated to $B$.  Note that
$$A^2=\begin{pmatrix}BB^\top&0\\0&B^\top B\\ \end{pmatrix}$$
has square diagonal blocks.


\begin{lem}\label{norm.is.max.of.block.norms}
If $M\in{\rm Mat}_{m+n}(\R)$ is block diagonal on square matrices $M_1\in{\rm Mat}_m(\R)$ and $M_2\in{\rm Mat}_n(\R)$,
then $|M|={\rm max}(|M_1|,|M_2|)$.
\end{lem}

\begin{proof}
Let $m={\rm max}(|M_1|,|M_2|)$.  So $|M_1|\leq m$ and $|M_2|\leq m$.
So $|M_1v_1|\leq m|v_1|$ and $|M_2v_2|\leq m|v_2|$ for any $v_1\in\R^m$ and any $v_2\in\R^n$. Consider 
$$v=\begin{pmatrix}v_1\\v_2\\ \end{pmatrix}\in\R^{m+n}\quad {\rm and}\quad
Mv=\begin{pmatrix}M_1&0\\0&M_2\\ \end{pmatrix}\begin{pmatrix}v_1\\v_2\\ \end{pmatrix}=\begin{pmatrix}M_1v_1\\M_2v_2\\ \end{pmatrix}$$
So $|Mv|^2=|M_1v_1|^2+|M_2v_2|^2\leq m^2(|v_1|^2+|v_2|^2)=m^2|v|^2$.  Thus
$|M|\leq m$.   Suppose $|M_1|\geq |M_2|$ so that $|M_1|=m$.
There is a unit vector $v_1$ such that $|M_1v_1|=|M_1|$.  Taking $v_2=0$
in the above gives unit vector $v$ with $|Mv|=|M_1v_1|=|M_1|=m$.  So $|M|\geq m$, and
thus $|M| = m$.
\end{proof}

For any $B$ consider the symmetric matrix $A$ associated to it.  Lemma \ref{norm.is.max.of.block.norms} can be used to prove the following result.

\begin{thm}
For an operator with matrix $B$, its operator norm is equal to the operator norm of its associated symmetric matrix $A$: $|A|=|B|$.
\end{thm}

\begin{proof}
For the square matrix $A$ associated to $B$ shown above, $A^2$ has square diagonal blocks $M_1=BB^\top$ and $M_2=B^\top B$.  For any vector $v$ with $|v|=1$,
$$|Bv|^2=\langle Bv,Bv\rangle=\langle B^\top Bv,v\rangle\leq|B^\top Bv|\leq|B^\top B|.$$
This gives $|Bv|\leq\sqrt{|B^\top B|}$, so that $|B|^2\leq|B^\top B|\leq|B^\top|\cdot|B|$.  Thus, for $B\neq 0$, $|B|\leq|B^\top|$.  Replacing $B$ with $B^\top$ in the above argument gives $|B^\top|\leq|B|$, so that $|B|=|B^\top|$.  So
$$|B^\top|\cdot|B|=|B|^2\leq|B^\top B|\leq|B^\top|\cdot|B|,$$
which gives $|B|^2=|B^\top B|$.  Replacing $B$ with $B^\top$ gives $|B|^2=|BB^\top|$.
Now $|M_1|=|BB^\top|=|B|^2$ and $|M_2|=|B^\top B|=|B|^2$, so $|A^2|=|B|^2$ by Lemma \ref{norm.is.max.of.block.norms}.
But $A^2=A^\top A$ and $|A^\top A|=|A|^2$. So $|A|^2=|B|^2$ and thus $|A|=|B|$.
\end{proof}



So the study of small operators can be carried out in the setting of small symmetric matrices.
The set of eigenvalues of a square matrix  $A$ is 
called its spectrum; and the spectral radius $\rho(A)$ is the radius of 
the smallest disk centered at 0 in the complex plane and containing the spectrum of $A$.
When $A$ is a symmetric matrix, all its eigenvalues are real numbers,
and $\rho(A)$ is the largest of these in absolute value, leading to the following well-known
result.

\begin{thm}\label{norm.equals.spectral.radius}
For a symmetric matrix $A$, the operator norm $|A|$ is equal to the spectral radius $\rho(A)$.
\end{thm}

\begin{proof}
A symmetric $n\times n$ matrix $A$ determines an orthonormal basis of eigenvectors $v_1,\ldots, v_n$, 
with real eigenvalues $\lambda_1,\ldots,\lambda_n$, such that $|\lambda_i|\leq |\lambda_n|$ for all $i$.  
So $\rho(A)=|\lambda_n|$, and $|A|\geq |\lambda_n|$ since $|Av_n|=|\lambda_n|$.
Let $v=\sum a_i v_i$ with $|v|^2=\sum a_i^2=1$; then $|A|\leq |\lambda_n|$, since
$$|Av|^2=\left|\sum a_i \lambda_iv_i\right|^2=\sum |a_i |^2|\lambda_i|^2 \leq \left(\sum |a_i |^2\right)|\lambda_n|^2=|\lambda_n|^2.$$
\end{proof}

\section{Non-negative square matrices}

For a matrix  $B$ with real entries, we write $B\geq 0$ when all the entries of $B$ are non-negative, and say that $B$ is non-negative. 
When $B\geq 0$ and some entry is non-zero, we write $B>0$; 
when $B\geq 0$ and all entries are non-zero (positive), we write $B\gg 0$.
Let $A\geq B$ mean $A-B\geq 0$, and $A>B$ mean $A-B>0$, and $A\gg B$ mean $A-B\gg 0$. 
In particular, the above notations apply for vectors with real entries.

We want to use some part of the Perron-Frobenius theory of non-negative square matrices.  
The proofs in this theory tend to be rather intricate; see Gantmacher \cite{Gant} or Sternberg \cite{Stern}, for instance.
But we only need to consider symmetric matrices in this paper.  So we present proofs for the symmetric case;
they seem simpler than the usual proofs, making efficient use of the 
Rayleigh quotient function for a symmetric matrix.

Consider $\R^n$ with its real inner product (dot product) $\langle x,y \rangle=x^\top y$.
For a square symmetric matrix $A$, 
define the real-valued function $R_A$  by 
$$R_A(x)={\frac{\langle Ax,x\rangle}{\langle x,x\rangle}}$$ 
for $x\neq 0$.  Note that $R_A(ax)=R_A(x)$ for any non-zero number $a$; so we can consider $R_A$ to be defined on the set of rays,
or on the set of unit vectors.  Let $\lambda$ denote the maximal value achieved by $R_A$ on the unit sphere.  
The fact that $\lambda$ is the largest
of all the eigenvalues of $A$ is
part of the ``minmax principle'' for the Rayleigh quotient of $A$ (see \cite{lax}).
Recall that a symmetric matrix has all its eigenvalues real.

\begin{thm}\label{Rayleigh.max}
The maximum value of the Rayleigh quotient  $R_A$ is the largest eigenvalue $\lambda$ of $A$,
and the maximum is achieved only at eigenvectors for $A$ and $\lambda$.
\end{thm}

\begin{proof}
A symmetric $n\times n$ matrix $A$ determines an orthonormal basis of eigenvectors $v_1,\ldots, v_n$, with real eigenvalues
$\lambda_1,\ldots,\lambda_n$, such that $\lambda_i\leq\lambda_n$ for all $i$.  
Let $v=\sum a_i v_i$ with $|v|^2=\sum a_i^2=1$; then
$$R_A(v)=\left\langle \sum a_i\lambda_i v_i,\sum a_i v_i\right\rangle=  \sum a_i^2 \lambda_i \leq \left(\sum a_i^2\right) \lambda_n=\lambda_n.$$
Also, $R_A(v_n)=\lambda_n$, so  the maximum value of $R_A(v)$ is $\lambda_n$.
If $R_A(v)=\lambda_n$ then $\sum (\lambda_n-\lambda_i)a_i^2=0$, all non-negative, so that we must have
$(\lambda_n-\lambda_i)a_i=0$ for all $i$; then  $\sum a_i\lambda_i v_i=\sum a_i\lambda v_i$, and $Av=\lambda_nv$.
\end{proof}

\medskip

For any vector $x$, let $\abs(x)$ denote the vector whose entries are the absolute values of the entries of $x$.
Note that $ \langle\abs(x), \abs(x)\rangle= \langle x, x\rangle$.  Also, if $A\geq 0$ then
$\abs(A\ x)\leq A\ \abs(x)$, by the triangle inequality.

\begin{thm}[{\bf Non-Negative Eigenvector, Symmetric Case}]\label{eigenvector.thm}
If $A>0$ and $A$ is symmetric, then the maximum value $\lambda$ of $R_A$ is achieved at some $z>0$ with $Az=\lambda z$.  Also, $|\lambda'|\leq\lambda$ for every eigenvalue $\lambda'$ of $A$.
\end{thm}

\begin{proof}
 Assume $A>0$ and $A$ symmetric.  Let $\lambda$ be the maximum value of $R_A$ on the unit sphere, achieved at $x$.
We have $\lambda>0$, since $R_A(e)>0$ where $e$ is the vector of $1$'s.  Also, $A x=\lambda x$ by Theorem \ref{Rayleigh.max}.

Now let $\lambda'$ be any eigenvalue of $A$, with $Ay=\lambda'y$ and $|y|=1$.
Apply the $\abs$ operator to $Ay=\lambda' y$, and use $\langle\abs(y), \abs(y)\rangle= \langle y, y\rangle=1$
to get:
$$|\lambda'|\ \abs(y)=\abs(\lambda' y)=\abs(A y)\leq A\ \abs(y)\quad {\rm so}\quad$$
$$|\lambda'|=\langle|\lambda'| \ \abs(y), \abs(y)\rangle \leq  \langle A \ \abs(y),  \abs(y)\rangle=R_A(\abs(y))\leq \lambda,$$
since $\lambda$ is the maximum of $R_A$. Thus  $|\lambda'|\leq \lambda$, and
for the eigenvalue $\lambda>0$ with $Ax=\lambda x$, we have $\lambda\leq R_A(\abs(x))\leq\lambda$.  By Theorem \ref{Rayleigh.max}, $z=\abs(x)$ is an eigenvector for $A$ with eigenvalue $\lambda$, and $z> 0$ since $x\neq 0$.

\end{proof}

\begin{rem}
Note that for a symmetric matrix $A$ with $A>0$, we have
$$\lambda=\max_{|x|=1}{R_A(x)}=\rho(A)=|A|$$
according to Theorems \ref{norm.equals.spectral.radius}, \ref{Rayleigh.max}, and \ref{eigenvector.thm}.
\end{rem}

Any square matrix has an associated matrix of 1's and 0's, where 1 means non-zero;
and we may interpret this matrix of 1's and 0's as the adjacency matrix
of a {\it directed} graph.  A directed graph is strongly connected if it contains a directed path
from each vertex to every other vertex.
Following Frobenius, 
let us say that a non-negative square matrix is irreducible
when its underlying directed graph is strongly connected.  
 
When $A$ is irreducible, there exists a square matrix $P$ with
all its entries non-zero and with $AP=PA$.
In fact, since the directed graph
 is strongly connected,  we can choose an integer $N$ so large that
there exists a path of length at most $N$ from each vertex to every other vertex. 
Then in the directed graph for matrix $I+A$,  each vertex has a path of length $N$ to each other vertex.
But the entries of $(I+A)^N$ count the paths of length $N$ in this  directed graph-with-loops;
thus $P=(I+A)^N$ for large $N$ has the desired properties.

Since we consider symmetric matrices, we don't need to consider the
underlying graph as a directed graph; and it is strongly connected if and only if it is connected.   

\begin{thm}[{\bf Spectral Radius, Symmetric Case}]\label{pos.eigenvector} If $A$ is irreducible, $A>0$, and $A$ is symmetric, 
then there exists $y\gg 0$ with $Ay=\rho(A) y$. 
\end{thm}

\begin{proof}
Since $A$ is irreducible and $A>0$, $P=(I+A)^N$ for large $N$ gives 
a (symmetric square) matrix $P\gg 0$ with $AP=PA$.
Moreover, there exists $z>0$ with $Az=\lambda z=\rho(A)z$, as in Theorem \ref{eigenvector.thm}.
Then $APz=PAz=\lambda Pz=\rho(A)Pz$, and $Pz\gg 0$.  So let $y=Pz$.  
\end{proof}


\begin{thm}[{\bf Comparison Theorem, Symmetric Case}]\label{comparison.theorem}
If $A$ is irreducible and symmetric and $A>B>0$, then $\rho(A)>\rho(B)$.
\end{thm}

\begin{proof}
We start by applying Theorem \ref{eigenvector.thm} to $B$. Let $R_B$ achieve its maximum value $\mu>0$ at unit vector $y>0$.
Then by Theorems \ref{Rayleigh.max} and \ref{eigenvector.thm}, $\mu=\rho(B)$.  Since $A-B>0$ and $y>0$, $(A-B)y\geq0$.  This gives $\langle(A-B)y,y\rangle\geq0$, so that $R_A(y)-R_B(y)\geq 0$.

Now apply Theorems \ref{Rayleigh.max}, \ref{eigenvector.thm}, and \ref{pos.eigenvector} to $A$, with $R_A$ achieving its maximum value $\lambda=\rho(A)$ at unit vector $x\gg 0$.  Then
$\lambda=R_A(x)\geq R_A(y)\geq R_B(y)=\mu$.  So $\lambda\geq \mu$.

We now show $\lambda\neq \mu$.
Suppose $\lambda=\mu$; then $R_A$ achieves its maximum value at the unit vectors $x\gg 0$ and $y>0$.
Suppose $x\neq y$; then $x$ and $y$ are linearly independent.  Let $c=\max\left\lbrace y_i/x_i\right\rbrace$.  Then $z=cx-y>0$, $Az=\lambda z$,
and $z$ has some entry $z_i=0$ and some entry $z_j > 0$.  But $A$ is irreducible, so
the underlying graph of $A$ has a path from vertex $i$ to vertex $j$, say of length $m$; then
$A^m$ has its $(i,j)$ entry non-zero. Then $z'=A^m z$ has $z'_i\geq(A^m)_{ij}z_j>0$;
but this contradicts $z'_i=\lambda^m z_i=0$, which follows from $A^m z=\lambda^m z$.
Thus, $x=y$ and $(A-B)x=(\lambda-\mu)x=0$, which is impossible since $(A-B)>0$ and $x\gg 0$ implies  $(A-B)x>0$.
Thus, $\lambda>\mu$, i.e., $\rho(A)>\rho(B)$.



\end{proof}

\section{Graphs}


A graph is a finite set of vertices and edges.  Let's exclude loops and multiple edges.
We say that vertices $x$ and $y$ are adjacent when $xy$ is an edge.
Enumerating the vertices of a graph gives an adjacency matrix which completely describes 
the graph; it is a symmetric matrix of 0's and 1's, indicating which vertices are adjacent.

A graph is bicolored if we have assigned a color red or blue to each vertex, so that each edge
connects a red and a blue vertex. 
A bicolored graph is completely described by an $m\times n$ matrix $B$ of 0's and 1's, once
we enumerate its $m$ red vertices and its $n$ blue vertices.

A small operator corresponds to a small matrix of 0's and 1's, which in turn corresponds
to a small bicolored graph. 
If $B$ is an $m\times n$ matrix corresponding to a bicolored graph, 
then the adjacency matrix of the underlying graph (forgetting the bicoloring)
is the symmetric matrix $A$ associated to $B$.


We will classify the small bicolored graphs, up to 
isomorphism and disjoint union of bicolored graphs. 
The first step (which turns out to be the main step for the classification of small operators) 
is the classification of ``small'' graphs, in the following sense.

Let us say that a graph is small when its adjacency matrix has spectral radius less than 2.
So we have that the matrix of a small operator corresponds exactly to a bicoloring of
a small graph.


Equivalence of graphs is isomorphism of graphs.
Decomposition of graphs is disjoint union of graphs.
A graph is small if and only if all its connected components are small graphs.

Now we use our results about non-negative square matrices from the previous section.
Recall that the undirected graphs that we work with have symmetric adjacency matrix,
which is irreducible if and only if the graph is connected.

In particular, if $A'$ is the adjacency matrix of a proper subgraph 
of a connected graph with adjacency matrix $A$,
then $A>A'\geq 0$ and $\rho(A)>\rho(A')$.

When $A$ is the $n\times n$ adjacency matrix for one of our graphs,
we can interpret an $n\times 1$ vector $v$ 
as assigning a number $v(x)$ to each vertex $x$ of the graph.
Then $v'=Av$ means that $v'(x)=\sum v(y)$, where we sum over the vertices $y$ which are 
adjacent to $x$.

This helps us verify that the connected graphs in Figure \ref{forbidden.subgraphs} all have spectral radius $2$: just assign a number $v(x)$ to each vertex $x$ 
so that $2v(x)=\sum v(y)$, where we sum over the vertices $y$ which are 
adjacent to $x$.
We refer to these graphs as ``forbidden subgraphs'', since a {\it small} connected graph 
cannot contain any of these as a subgraph (and still have spectral radius less than $2$).

\begin{thm}[{\bf Classification for Small Graphs}]
A small graph is a disjoint union of connected small graphs.  Each connected small graph is isomorphic to $A_n$, $D_n$, or $E_n$, for some $n$.
\end{thm}

\begin{proof}
Let $\Gamma$ be a connected small graph.  Note that $\Gamma$ cannot contain a cycle, a vertex of degree $4$ or more, or more than one vertex of degree $3$ (since then $\Gamma$ would contain one of the forbidden subgraphs in Figure \ref{forbidden.subgraphs}, and would have spectral radius at least $2$ by Theorem \ref{comparison.theorem}).  Let $T_{p,q,r}$ denote the ``tripod'' graph, consisting of three legs with $p$, $q$, and $r$ vertices.  If $\Gamma$ is a small tripod graph, there are limitations on how long its legs can be (since $\Gamma$ cannot contain any of the forbidden tripods in Figure \ref{forbidden.subgraphs} as subgraphs).  Thus, a small connected graph is isomorphic to one from the $ADE$ series, shown in Figure \ref{small.graphs}.  Note that each of the graphs in Figure \ref{small.graphs} is a proper subgraph of a graph in Figure \ref{forbidden.subgraphs}, and so must have spectral radius less than $2$ by Theorem \ref{comparison.theorem}.

Note that if $\Gamma$ is small but not connected, the vertices can be enumerated so that its adjacency matrix is block diagonal on square matrices.  Using Lemma \ref{norm.is.max.of.block.norms}, each connected component of $\Gamma$ must be small so that $\Gamma$ is a disjoint union of connected small graphs.
\end{proof}

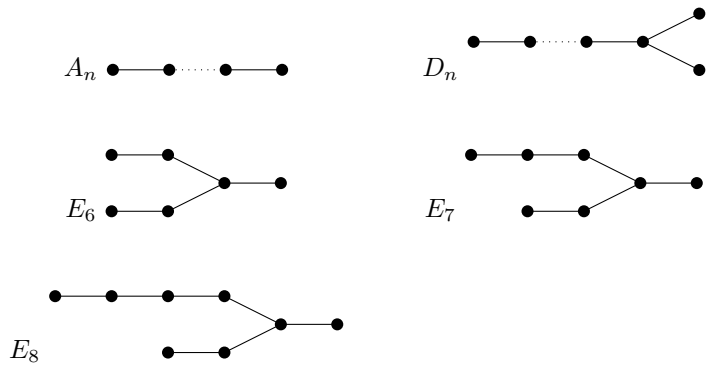
\begin{figure}[h]
\caption{Forbidden subgraphs, each with spectral radius $2$}
\label{forbidden.subgraphs}
\phantom{}
\begin{center}
\begin{tabular}{cccc}
\begin{dynkin}
  \dynkinline{1}{0}{2}{.5};
  \dynkinline{1}{0}{2}{-.5};
  \dynkinline{2}{.5}{3}{.5};
  \dynkinline{2}{-.5}{3}{-.5};
  \dynkindots{3}{-.5}{3}{.5};
  \dynkindot{1}{0};
  \dynkindot{2}{.5};
  \dynkindot{2}{-.5};
  \dynkindot{3}{.5};
  \dynkindot{3}{-.5};
\end{dynkin}

&&&
\begin{dynkin}
	\dynkinline{1}{.5}{2}{0};
    \dynkinline{1}{-.5}{2}{0};
    \dynkinline{2}{0}{3}{.5};
    \dynkinline{2}{0}{3}{-.5};
    \dynkindot{1}{.5};
    \dynkindot{1}{-.5};
     \dynkindot{3}{.5};
    \dynkindot{3}{-.5};
    \foreach \x in {2,...,2}
{
       \ifnum \x=100 {\dynkincross{\x}{0}}
       \else {\dynkindot{\x}{0}}
       \fi
    }
\end{dynkin}\\

&&&\\
&&&\\

\begin{dynkin}
	\dynkinline{1}{.5}{2}{0};
    \dynkinline{1}{-.5}{2}{0};
    \dynkinline{2}{0}{3}{0};
    \dynkindots{3}{0}{4}{0};
    \dynkinline{4}{0}{5}{0};
    \dynkinline{5}{0}{6}{.5};
    \dynkinline{5}{0}{6}{-.5};
    \dynkindot{1}{.5};
    \dynkindot{1}{-.5};
     \dynkindot{6}{.5};
    \dynkindot{6}{-.5};
    \foreach \x in {2,...,5}
{
       \ifnum \x=100 {\dynkincross{\x}{0}}
       \else {\dynkindot{\x}{0}}
       \fi
    }
\end{dynkin}

&&&  
\begin{dynkin}
    \dynkinline{1}{.5}{2}{.5};
    \dynkinline{2}{.5}{3}{0};
    \dynkinline{1}{-.5}{2}{-.5};
    \dynkinline{2}{-.5}{3}{0};
    \dynkinline{3}{0}{4}{0};
    \dynkinline{4}{0}{5}{0};
    \dynkindot{1}{.5};
    \dynkindot{1}{-.5};
    \dynkindot{2}{.5};
    \dynkindot{2}{-.5};
    \foreach \x in {3,...,5}
{
       \ifnum \x=100 {\dynkincross{\x}{0}}
       \else {\dynkindot{\x}{0}}
       \fi
    }
\end{dynkin}\\

&&&\\
&&&\\

\begin{dynkin}
    \dynkinline{1}{.5}{2}{.5};
    \dynkinline{2}{.5}{3}{.5};
    \dynkinline{3}{.5}{4}{0};
    \dynkinline{4}{0}{5}{0};
    \dynkinline{1}{-.5}{2}{-.5};
    \dynkinline{2}{-.5}{3}{-.5};
    \dynkinline{3}{-.5}{4}{0};
    \dynkindot{1}{.5};
    \dynkindot{2}{.5};
    \dynkindot{3}{.5};
    \dynkindot{1}{-.5};
    \dynkindot{2}{-.5};
    \dynkindot{3}{-.5};
    \foreach \x in {4,...,5}
{
       \ifnum \x=100 {\dynkincross{\x}{0}}
       \else {\dynkindot{\x}{0}}
       \fi
    }
\end{dynkin}

&&&

\begin{dynkin}
		\dynkinline{0}{.5}{1}{.5};
    \dynkinline{1}{.5}{2}{.5};
    \dynkinline{2}{.5}{3}{.5};
    \dynkinline{3}{.5}{4}{.5};
    \dynkinline{4}{.5}{5}{0};
    \dynkinline{5}{0}{6}{0};
    \dynkinline{3}{-.5}{4}{-.5};
    \dynkinline{4}{-.5}{5}{0};
    \dynkindot{5}{0};
    \dynkindot{6}{0};
    \dynkindot{3}{-.5};
    \dynkindot{4}{-.5};
    \foreach \x in {0,...,4}
{
       \ifnum \x=100 {\dynkincross{\x}{0}}
       \else {\dynkindot{\x}{.5}}
       \fi
    }
\end{dynkin}\\

\end{tabular}
\end{center}
\end{figure}

\begin{figure}[h]
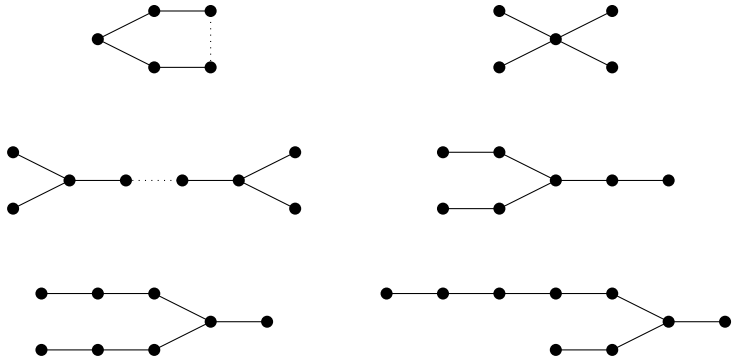

\caption{Connected graphs with spectral radius $<2$}
\label{small.graphs}
\phantom{}
\begin{center}
\begin{tabular}{cccc}
$A_n$~~
\begin{dynkin}
  \dynkinline{1}{0}{2}{0};
  \dynkindots{2}{0}{3}{0};
  \dynkinline{3}{0}{4}{0};
    \foreach \x in {1,...,4}
{
       \ifnum \x=100 {\dynkincross{\x}{0}}
       \else {\dynkindot{\x}{0}}
       \fi
    }
  \end{dynkin}

&&&

$D_n$~~
\begin{dynkin}
    \dynkinline{1}{0}{2}{0};
    \dynkindots{2}{0}{3}{0};
    \dynkinline{3}{0}{4}{0};
    \dynkinline{4}{0}{5}{.5};
    \dynkinline{4}{0}{5}{-.5};
    \dynkindot{5}{.5};
		\dynkindot{5}{-.5};
    \foreach \x in {1,...,4}
{
       \ifnum \x=100 {\dynkincross{\x}{0}}
       \else {\dynkindot{\x}{0}}
       \fi
    }
  \end{dynkin}\\
	
&&&\\
&&&\\

$E_6$~~
\begin{dynkin}
    \dynkinline{1}{.5}{2}{.5};
    \dynkinline{2}{.5}{3}{0};
    \dynkinline{3}{0}{4}{0};
    \dynkinline{1}{-.5}{2}{-.5};
    \dynkinline{2}{-.5}{3}{0};
    \dynkindot{1}{.5};
    \dynkindot{2}{.5};
    \dynkindot{1}{-.5};
    \dynkindot{2}{-.5};
    \foreach \x in {3,...,4}
{
       \ifnum \x=100 {\dynkincross{\x}{0}}
       \else {\dynkindot{\x}{0}}
       \fi
    }
\end{dynkin}

&&&

$E_7$~~
\begin{dynkin}
    \dynkinline{1}{.5}{2}{.5};
    \dynkinline{2}{.5}{3}{.5};
    \dynkinline{3}{.5}{4}{0};
    \dynkinline{4}{0}{5}{0};
    \dynkinline{2}{-.5}{3}{-.5};
    \dynkinline{3}{-.5}{4}{0};
    \dynkindot{4}{0};
    \dynkindot{5}{0};
    \dynkindot{2}{-.5};
    \dynkindot{3}{-.5};
    \foreach \x in {1,...,3}
{
       \ifnum \x=100 {\dynkincross{\x}{0}}
       \else {\dynkindot{\x}{.5}}
       \fi
    }
\end{dynkin}\\

&&&\\
&&&\\

$E_8$~~
\begin{dynkin}
    \dynkinline{0}{.5}{1}{.5};
    \dynkinline{1}{.5}{2}{.5};
    \dynkinline{2}{.5}{3}{.5};
    \dynkinline{3}{.5}{4}{0};
    \dynkinline{4}{0}{5}{0};
    \dynkinline{2}{-.5}{3}{-.5};
    \dynkinline{3}{-.5}{4}{0};
    \dynkindot{4}{0};
    \dynkindot{5}{0};
    \dynkindot{2}{-.5};
    \dynkindot{3}{-.5};
    \foreach \x in {0,...,3}
{
       \ifnum \x=100 {\dynkincross{\x}{0}}
       \else {\dynkindot{\x}{.5}}
       \fi
    }
\end{dynkin}

&&&\\
\end{tabular}
\end{center}
\end{figure}

\vfill\newpage
\begin{thm}[{\bf Classification Theorem for Small Operators}]
A small operator corresponds to a small bicolored graph, which is isomorphic to
a disjoint union of connected small graphs together with a bicoloration.
\end{thm}

\section{Remarks}

Here are some brief remarks on some famous classification results from different areas of mathematics
Each of these areas has a natural notion of decomposition into indecomposables, 
and the classification is largely 
parallel to the classification of small operators.

Let's start with a historical sketch of the classification story.

In the 1880's Wilhelm Killing worked on classifying possible types of geometries.
He used recent 
developments in linear algebra to work out a classification 
of (what turned out to be) the semi-simple Lie algebras over the complex numbers.
He was partly inspired by Sophus Lie's ongoing work on ``continuous groups''.

In particular, Killing used sophisticated ideas about eigenvalues to
record an isomorphism class of semi-simple
Lie algebras in terms of a ``root system''.
The root systems and  Lie algebras are then 
built up as direct sums of indecomposable ones; and Killing
essentially classified the indecomposables
into types $A$, $B$, $C$, $D$, $E$, $F$, and $G$.


Eli Cartan organized and completed this classification in his 1894 thesis. 
The data for a root system can be encoded in a matrix of integers,
now called the Cartan matrix of the root system;
it determines an integer-valued bilinear form on a 
maximal abelian subalgebra of the Lie algebra.
See Coleman \cite{Coleman} and Hawkins \cite{Hawkins} for more on the history of these developments. 

Donald Coxeter made a separate contribution through his  study of kaleidoscopes. 
By the early 1930's he had classified
those sets of mirrors in a real finite-dimensional inner product space
which generate a finite group of reflections (see \cite{roberts}).
In his 1934 paper ``Discrete groups generated by reflections'' \cite{Cox},
Coxeter used graphs to describe his mirror systems.
In particular, certain of his finite reflection groups were encoded by
connected undirected graphs, without loops and multiple edges.  
We refer to these as the $ADE$ series of graphs.

Hermann Weyl gave a series of lectures on Lie algebras at Princeton that year,
and Coxeter observed that his crystallographic reflection groups (those preserving a lattice)  
correspond to certain permutation groups of roots in a root system, now called
the Weyl groups. The connected $ADE$ graphs determine 
the Cartan matrices for the ``simply-laced'' simple Lie algebras.
The mimeographed lecture notes, published 1934-1935, include an appendix by Coxeter in which
these graphs appear \cite{Weyl}.
These classification ideas continued to be developed in work
by Eugene Dynkin (1947, 1952),
Bourbaki (1968, with exposition attributed to Jacques Tits), and many others.

The book by Fulton and Harris \cite{FH} is one good reference for the theory.



\bigskip
Let us close with a presentation of these $ADE$ classification results, organized around our notion of small graph.

\bigskip\noindent
{\bf A graph determines a quadratic form:}
Let $X$ be a graph with vertex set $X_0$.   Let ${\Z}X_0$ be the free abelian group with basis $X_0$,
so that the elements of ${\Z}X_0$ are the integer linear combinations of the elements in $X_0$.
Define  an integer-valued symmetric bilinear form on ${\Z}X_0$ by describing its values on $X_0$:
\[ (x|x)=2, \quad{\rm and}\quad  
   (x|y)=\begin{cases}  -1, &\text{if $xy$ is an edge;}\\
				\ \ 0, &\text{if $xy$ is not an edge.}\\
	\end{cases}\]

This bilinear form is ``even'', in that $q_X(v)=(v|v)/2$ defines an integral-valued quadratic form on ${\Z}X_0$.  This means $q_X(v+w)=q_X(v)+q_X(v|w)+q_X(w)$.  Note that over the integers, it is more convenient to not include the usual factor of $2$ in the middle correction term.

\begin{thm}
A graph $X$ is small if and only if its quadratic form $q_X$ is positive-definite.
\end{thm}

\begin{proof}
Let $A$ be the adjacency matrix of the graph; so $X$ is small if and only if $\rho(A)<2$, if and only if $C=2I-A$ is positive definite.
But $C$ is the symmetric matrix recording the  bilinear form corresponding to $q_X$.
\end{proof}

From this perspective, 
an integral-valued quadratic form on ${\Z}X_0$ is {\it small} when it is positive definite, and we have classified
the small quadratic forms.


\bigskip\noindent
{\bf A graph determines a group:}  For each vertex $x$, define an additive involution 
$s_x:{\Z}X_0\to {\Z}X_0$  by describing its values on $X_0$:
\[ s_x(x)=-x, \quad{\rm and}\quad  
   s_x(y)=\begin{cases}  	y+x, &\text{if $xy$ is an edge;}\\
					y, &\text{if $xy$ is not an edge.}\\
	\end{cases}\]

If $v=\sum_x v_x x$ in $\Z X_0$, we have $s_x(v)=v'$ where $v'_x=-v_x+\sum v_y$ where the sum is over those
vertices $y$ which are adjacent to $x$,  
and $v'_z=v_z$ for $z\neq x$.
So $s_x$ replaces the label 
at vertex $x$ by the sum of surrounding labels, minus the original label.  
In terms of the bilinear form for the graph $X$, the involution $s_x$  associated to vertex $x$ is given by
$v\mapsto v-(v|x)x$. 
Let $G_X$ be the group of additive isomorphisms of ${\Z}X_0$ generated by the $s_x$.
The group $G_X$ preserves the quadratic form $q_X$.


\begin{thm}
The graph $X$ is small if and only if the group $G_X$ is finite.
\end{thm}

\begin{proof}
From the definition, $s_x s_x=1$ for every vertex $x$, $s_x s_y s_x=s_y s_x s_y$ if $xy$ is an edge, and $s_y s_x=s_x s_y$ if $xy$ not an edge.  This means that $s_y s_x$ has order 2 if $xy$ not an edge, and $s_y s_x$ has order 3 if $xy$ is an edge.  This establishes the connection between graphs and the presentation of the simply-laced crystallographic Coxeter groups.  For more details, see Coxeter's paper \cite{Cox}.
\end{proof}

The small graphs correspond to the simply-laced Weyl groups, so the above work 
completes the classification of the simply-laced Weyl groups.

Let us go on to explain the notion of ``root system'' associated to these ideas.

\bigskip\noindent
{\bf Graphs and (simply-laced) root systems:}  A graph $X$ determines a group $W=G_X$ 
and a lattice $\Lambda=\Z X_0$, 
together with a quadratic form $q=q_X$ with $q(x)=1$ for $x\in X_0$.
Thus the set $\Gamma=\{v\in \Lambda: q(v)=1\}$
generates $\Lambda$ as a $\Z$-module. 

The graph is small if and only if $q$ is positive definite.
For such a triple $(\Lambda,q,\Gamma)$ the real vector space $V$ generated by  $\Gamma$ (the set of roots) is 
an inner product space with norm $q$, and $\Gamma$ is finite, since $\Gamma$ is the
the intersection of a lattice and the unit sphere in a Euclidean space.
Then $\Lambda$  is called the root lattice in this Euclidean space, and $\Gamma$ is called the set of roots.
Moreover,  the involutions $s_x$, which generate $W$ are orthogonal reflections in this Euclidean space.

These are the root systems of the simply-laced semi-simple Lie algebras over the complex numbers; 
see Lurie's discussion in \cite{Lurie}. 

\bigskip
This is the $ADE$ classification result:

\begin{thm}
Simply-laced root systems are classified by small graphs.
\end{thm}


\vfill\newpage
\bibliographystyle{amsplain}
\bibliography{bibliography}

\providecommand{\bysame}{\leavevmode\hbox to3em{\hrulefill}\thinspace}
\providecommand{\MR}{\relax\ifhmode\unskip\space\fi MR }
\providecommand{\MRhref}[2]{%
  \href{http://www.ams.org/mathscinet-getitem?mr=#1}{#2}
}
\providecommand{\href}[2]{#2}
\begin{thebibliography}{10}

\bibitem{symposium}
Vladimir Arnol'd, \emph{Problems of present day mathematics}, Mathematical
  developments arising from {H}ilbert problems (Felix~E Browder, ed.), Proc.
  Symp. Pure Math., vol.~28, Amer. Math. Soc., 1976, p.~46.

\bibitem{cameron}
Peter~J Cameron, Jean-Marie Goethals, Johan~Jacob Seidel, and Ernest~E Shult,
  \emph{Line graphs, root systems, and elliptic geometry}, Journal of Algebra
  \textbf{43} (1976), no.~1, 305--327.

\bibitem{Coleman}
AJ~Coleman, \emph{The greatest mathematical paper of all time}, The
  Mathematical Intelligencer \textbf{11} (1989), no.~3, 29--38.

\bibitem{Cox}
Harold~SM Coxeter, \emph{Discrete groups generated by reflections}, Annals of
  Mathematics (1934), 588--621.

\bibitem{FH}
William Fulton and Joe Harris, \emph{Representation theory: a first course},
  vol. 129, Springer Science \& Business Media, 2013.

\bibitem{Gant}
Feliks~Ruvimovich Gantmacher and Joel~Lee Brenner, \emph{Applications of the
  theory of matrices}, Courier Corporation, 2005.

\bibitem{GdlaHJ}
F~Goodman, Pierre de~la Harpe, and Vaughan~FR Jones, \emph{Coxeter graphs and
  towers of algebras}, Mathematical Sciences Research Institute Publications,
  vol.~14, Springer-Verlag, New York, 1989.

\bibitem{Hawkins}
Thomas Hawkins, \emph{Emergence of the theory of {L}ie groups: An essay in the
  history of mathematics 1869--1926}, Springer Science \& Business Media, 2000.

\bibitem{HHSV}
Michiel Hazewinkel, Wim Hesselink, Dirk Siersma, and Ferdinand Veldkamp,
  \emph{The ubiquity of {C}oxeter {D}ynkin diagrams (an introduction to the
  {ADE} problem)}, Nieuw Archief voor Wiskunde \textbf{25} (1977), no.~3,
  257--307.

\bibitem{lax}
Peter~D Lax, \emph{Linear algebra}, Pure and Applied Mathematics,
  Wiley-Interscience, 1996.

\bibitem{Lurie}
Jacob Lurie, \emph{On simply laced {L}ie algebras and their minuscule
  representations}, Commentarii Mathematici Helvetici \textbf{76} (2001),
  no.~3, 515--575.

\bibitem{roberts}
Siobhan Roberts, \emph{King of infinite space: Donald {C}oxeter, the man who
  saved geometry}, Bloomsbury Publishing USA, 2009.

\bibitem{Stern}
Shlomo Sternberg, \emph{Dynamical systems}, Courier Corporation, 2010.

\bibitem{Weyl}
H.~Weyl and R.~Brauer, \emph{The structure and representation of continuous
  groups}, Lectures, Institute for Advanced Study, 1935.

\end{thebibliography}



\end{document}